\documentclass[graybox]{svmult}

\usepackage{type1cm}        
\usepackage{makeidx}         
\usepackage{graphicx}        
\usepackage{multicol}        
\usepackage[bottom]{footmisc}
\usepackage{algorithm,algorithmicx}
\usepackage{algpseudocode}
\usepackage[utf8]{inputenc}
\usepackage[T1]{fontenc}
\usepackage{newtxtext}       %
\usepackage{newtxmath}       
\usepackage[mathscr]{euscript}
\def\0{\mathbf{0}}
\def\1{\mathbf{1}}

\def\E{\mathds{E}}

\def\I{{\mathbf I}}
\def\L{{\mathbf L}}
\def\v{{\mathbf v}}

\def\W{{\mathbf W}}
\def\D{{\mathbf D}}

\def\R{{\mathbf R}}

\def\x{{\bf x}}

\def\u{{\bf u}}
\def\diag{\text{diag}}

\newfloat{Algorithm}{tbp}{lop}[section]

\makeindex             

\begin{document}

\title*{Dynamical Systems Theory and Algorithms for NP-hard Problems}
\author{Tuhin Sahai}
\institute{Tuhin Sahai \at Raytheon Technologies Research Center, 2855 Telegraph Ave. Suite 410, Berkeley, CA 94705 USA. \email{tuhin.sahai@gmail.com}}
%
%
\maketitle

\abstract*{This article surveys the burgeoning area at the intersection of dynamical systems theory and algorithms for NP-hard problems. Traditionally, computational complexity and the analysis of non-deterministic polynomial-time (NP)-hard problems have fallen under the purview of computer science and discrete optimization. However, over the past few years, dynamical systems theory has increasingly been used to construct new algorithms and shed light on the hardness of problem instances. We survey a range of examples that illustrate the use of dynamical systems theory in the context of computational complexity analysis and novel algorithm construction. In particular, we summarize a) a novel approach for clustering graphs using the wave equation partial differential equation,   b) invariant manifold computations for the traveling salesman problem, c) novel approaches for building quantum networks of Duffing oscillators to solve the MAX-CUT problem, d) applications of the Koopman operator for analyzing optimization algorithms, and e) the use of dynamical systems theory to analyze computational complexity. }

\abstract{This article surveys the burgeoning area at the intersection of dynamical systems theory and algorithms for NP-hard problems. Traditionally, computational complexity and the analysis of non-deterministic polynomial-time (NP)-hard problems have fallen under the purview of computer science and discrete optimization. However, over the past few years, dynamical systems theory has increasingly been used to construct new algorithms and shed light on the hardness of problem instances. We survey a range of examples that illustrate the use of dynamical systems theory in the context of computational complexity analysis and novel algorithm construction. In particular, we summarize a) a novel approach for clustering graphs using the wave equation partial differential equation,   b) invariant manifold computations for the traveling salesman problem, c) novel approaches for building quantum networks of Duffing oscillators to solve the MAX-CUT problem, d) applications of the Koopman operator for analyzing optimization algorithms, and e) the use of dynamical systems theory to analyze computational complexity. }
\vspace{0.2in}

\section{Introduction}
\label{sec:intro}
Dynamical systems theory and computational complexity have, predominantly, been developed as independent areas of research over the last century with little interaction and mutual influence. Dynamical systems theory has its origins in the seminal work of Henri Poincar\'e~\cite{Cit:Poincare} on celestial mechanics.  Computational complexity theory, on the other hand, originated in the works of Alan Turing~\cite{Cit:Turing} and Alonzo Church~\cite{Cit:Church} in the 1930s and has played an intimate role in the computing revolution of the twentieth century.

Eventually, dynamical systems theory (or nonlinear dynamics) found broad application beyond celestial mechanics. In particular, it has been used extensively to model and analyze engineering systems~\cite{Cit:Stro}, physics of natural phenomena, biological~\cite{Cit:Bio} and chemical processes~\cite{Cit:chem}, fluid dynamics~\cite{Cit:turb}, and epidemiology~\cite{Cit:epi} to name a few. Moreover, the analysis of dynamical systems is typically intimately tied to numerical methods~\cite{Cit:set_oriented,Cit:Igor} and scientific computation~\cite{Cit:scientific_comp}. 

Links to the applications (outlined in the previous paragraph) have played a critical role in the theoretical development of the field. For example, they have influenced the development of various sub-areas within nonlinear dynamics such as ergodicity~\cite{Cit:ergodicity}, chaos theory~\cite{Cit:Lorenz}, and symbolic dynamics~\cite{Cit:symb} to name a few. For a broad overview of the theoretical approaches to dynamical systems, we refer the reader to~\cite{Cit:Gucken}. Although, dynamical systems theory has found wide application in engineering and the sciences, it has received scant attention from the computer science community. 

Local continuous optimization techniques such as Nesterov's method~\cite{Cit:Nesterov} have recently been analyzed from a dynamical systems perspective~\cite{Cit:boyd_candes}. Nesterov's method is an optimal gradient descent algorithm in terms of convergence rate. In~\cite{Cit:boyd_candes}, the authors derive a dynamical system by invoking a continuous time limit of the optimization step size. They then analyze the resulting ordinary differential equations (ODEs) to provide valuable insight into the algorithm and its associated optimality. Additionally, in~\cite{Cit:wibisono} the authors use calculus of variations to gain additional insight into the convergence rates of accelerated gradient descent schemes. Although, this body of work does fall under the category of novel application of dynamical systems theory to optimization methods, we will not discuss it at length in this paper for two reasons a) this work has sparked extensive follow-on work and consequently, various summary articles and presentations are already available, and b) they appear to be restricted to accelerated gradient methods with no clear extension to the broader theory of computational complexity.

Non-deterministic polynomial-time (NP)-hard and -complete complexity classes can be traced to seminal work by Cook in 1971~\cite{Cit:cook}. The broad applicability of this work was outlined in a highly influential publication by Karp~\cite{Cit:karp}. NP-hard problems such as the traveling salesman problem (TSP)~\cite{Cit:TSP} and lattice-based vector problems~\cite{Cit:lattice} arise in a wide variety of applications ranging from DNA sequencing and astronomy~\cite{Cit:TSP_book} to encryption~\cite{Cit:lattice}. In essence, the computation of optimal solutions for these problems quickly becomes intractable with the size of the instance (unlike problems that lie in the $P$ complexity class). Note that some problems such as graph isomorphism~\cite{Cit:GI_klus} lie in the NP complexity class but are not expected to be NP-complete or NP-hard. Over the last few years, several efficient heuristic algorithms for approximating the solutions of NP-hard problems have been developed. For example, careful implementations of the Lin-Kernighan~\cite{Cit:LKH} and branch-and-bound~\cite{Cit:Concorde} heuristics have been successful in computing optimal solutions of several large instances of the TSP. However, most NP-hard problems suffer from a lack of scalable approaches. Moreover, as long as $P\neq NP$ (where $P$ is the complexity class of problems that can be solved in poylnomial time on a deterministic Turing machine), even efficient heuristics for some of these problems will remain elusive. See Fig.~\ref{fig:complexity_map} for the hypothesized relationship between the most popular classes.
\begin{figure}
    \centering
    \includegraphics[width=0.65\textwidth]{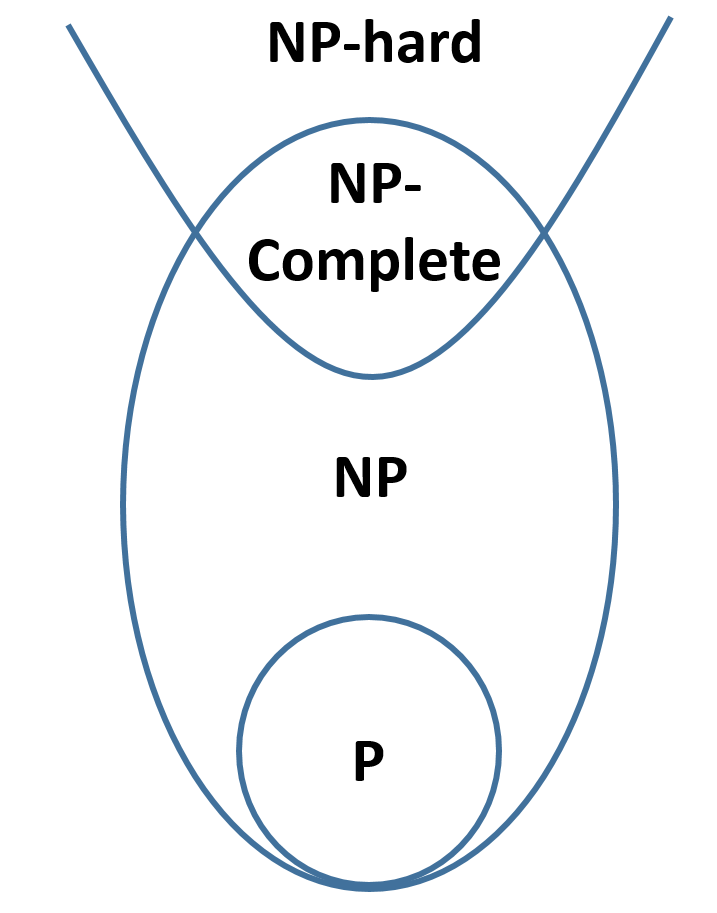}
    \caption{The computational complexity map for the most popular complexity classes.}
    \label{fig:complexity_map}
\end{figure}

In this work, we start by surveying the use of dynamical systems in the context of constructing state-of-the-art algorithms for NP-hard problems. In particular, we will cover the use of dynamical systems theory for constructing decentralized graph clustering algorithms~\cite{Cit:cluster1, Cit:cluster2}, solutions for the TSP~\cite{Cit:TSP_sahai}, and quantum-inspired networks of Duffing oscillators for solving the MAX-CUT problem~\cite{Cit:quantum_net}. We then switch to the use of dynamical systems theory for analysis of algorithms~\cite{Cit:koopman} and the underlying problems~\cite{Cit:zoltan,Cit:zoltan2}.

The goal of this survey paper is to highlight the potential application of dynamical systems theory for optimization of complex functions and analysis of computational complexity theory. This is a nascent field which presents the possibility of tremendous impact. Additionally, we expect this area to lead to new theoretical developments in nonlinear dynamics theory and novel algorithms for computationally intractable problems.

Ziessler, Surana, Speranzon, Klus, Dellnitz, and Banaszuk have all served as co-authors in my efforts in this area. However, my extensive discussions with Prof. Michael Dellnitz inspired me to delve deeper into the area of dynamical systems and the analysis of algorithms!

\section{Novel algorithm construction: decentralized graph clustering}
\label{sec:graph}
\begin{overview}{Overview}
Algorithms for graph analysis have a wide variety of applications such as routing, pattern recognition, database searches, network layout, and Internet PageRank to name a few~\cite{Cit:graph_algos}. Although some of these problems can solved efficiently on present day computing devices, several graph analysis problems are computationally intractable~\cite{Cit:cormen}. For example, the problem of partitioning graphs into equal size sets while minimizing the weights of cut edges arises in a range of settings such as social
anthropology, gene networks,
protein sequences, sensor
networks, computer
graphics, and Internet routing
algorithms~\cite{Cit:cluster2}.  To avoid unbalanced
cuts, size restrictions are typically placed on the clusters; instead of minimizing inter-connection strength, if one
minimizes the ratio of the inter-connection strength to the size
of individual clusters, the problem becomes
NP-complete~\cite{Cit:tutorial,Cit:npcomp}. 

In~\cite{Cit:cluster1,Cit:cluster2}, a novel decentralized algorithm for clustering/partitioning graphs that exploits fundamental properties of a dynamically evolving networked system was constructed. In particular, by propagating waves in a graph, one can compute partitions or clusters in a completely decentralized setting. The method is orders of magnitude faster than existing approaches~\cite{Cit:kempe}. This is our first example of a dynamical systems theory based algorithm for a combinatorial optimization problem. We now discuss the details of the approach.
\end{overview}
Let $\mathcal{G}=(V,E)$ be a graph with vertex set~$V =
\{1,\dots,N\}$ and edge set $E\subseteq V\times V$, where a
weight~$\W_{ij} \geq 0$ is associated with each edge $(i,j)\in
E$, and $\W$ is the $N\times N$ weighted adjacency matrix
of~$\mathcal{G}$. We assume that $\W_{ij}=0$ if and only if
$(i,j) \notin E$. The (normalized) graph Laplacian is defined
as,
\begin{align}
    \L_{ij} = \begin{cases}
                1 & \mbox{if}\: i = j\\
                -\W_{ij}/\sum_{\ell=1}^N \W_{i\ell} & \mbox{if}\: (i,j) \in E\\
                0   & \mbox{otherwise}\,,
             \end{cases}
             \label{eq:ldef}
\end{align}
or equivalently, $\L = \I-\D^{-1}\W$ where $\D$ is the
diagonal matrix with the row sums of $\W$.

Note that in~\cite{Cit:cluster2}, only undirected graphs were considered. The
smallest eigenvalue of the Laplacian matrix is $\lambda_1 = 0$,
with an associated eigenvector
$\v^{(1)}=\1=\left[1,1,\dots,1\right]^T$. Eigenvalues of~$\L$
can be ordered as, $ 0 = \lambda_1 \leq \lambda_2 \leq \lambda_3
\leq \cdots \leq \lambda_N$ with associated eigenvectors $\1,
\v^{(2)}, \v^{(3)}\cdots \v^{(N)}$~\cite{Cit:tutorial}. It is well
known that the multiplicity of~$\lambda_1$ is the number of
connected components in the graph~\cite{Cit:tutorial}. 

Given the Laplacian matrix~$\L$, associated with a
graph~$\mathcal{G}= (V,E)$, spectral clustering divides
$\mathcal{G}$ into two clusters by computing the signs of the~$N$
elements of the second eigenvector~$\v^{(2)}$, or Fiedler
vector. For further details about the computation of two or more clusters see~\cite{Cit:tutorial}. 

There are many algorithms to compute eigenvectors, such as the
Lanczos method or orthogonal iteration~\cite{GolubVanLoan96}.
Although some of these methods are distributable, convergence is
slow~\cite{GolubVanLoan96} and these algorithms do not
consider/take advantage of the fact that the matrix for which
the eigenvalues and eigenvectors need to be computed is the
adjacency matrix of the underlying graph.
In~\cite{Cit:kempe}, the authors propose an algorithm to
compute the first~$k$ largest eigenvectors (associated with the
first~$k$ eigenvalues with greatest absolute value)\footnote{
Note that in the case of spectral clustering we desire to
compute the smallest $k$ eigenvectors of~$\L$. The algorithm is
still applicable if we consider the matrix~$\I-\L$.} of a
symmetric matrix. The algorithm in~\cite{Cit:kempe}
emulates the behavior of orthogonal iteration using a decentralized process based on gossip algorithms or deterministic random walks on graphs. This approach can be slow as it converges after $O(\tau\log^2 N)$
iterations~\cite{Cit:kempe} where  where~$\tau$ is the mixing time for the
random walk on the graph and $N$ is the number of nodes. 

This procedure is equivalent
to evolving the discretized heat equation on the graph and can
be demonstrated as follows. The heat equation is given by,
\begin{align*}
    \frac{\partial u}{\partial t} = \Delta u\,,
\end{align*}
where~$u$ is a function of time and space,~$\partial u/\partial
t$ is the partial derivative of~$u$ with respect to time,
and~$\Delta$ is the Laplace operator~\cite{Cit:cluster2}.

When the
above equation is discretized on a graph
$\mathcal{G}=(V,E)$ one gets the following equation:
\begin{align*}
    \u_{i}(t+1) = \u_{i}(t) - \displaystyle\sum_{j\in\mathcal{N}(i)}\L_{ij}\u_{j}(t)\,,
\end{align*}
for~$i,j \in V$. Here~$\u_{i}(t)$ is the scalar value of~$u$ on
node~$i$ at time~$t$ and $\mathcal{N}(i)$ are the neighbors of node $i$ in the graph. The graph Laplacian $\L=[\L_{ij}]$
is the discrete counterpart of the $\Delta$
operator. The above iteration can be re-written,
in matrix form, $\u(t+1) = (\I-\L)\,\u(t)$ where~$\u(t) =
(\u_{1}(t),\dots,\u_{N}(t))^{T}$. The solution of this
iteration is,
\begin{align}
    \u(t) = C_0\1 + C_1 (1-\lambda_2)^t \v^{(2)} + \dots
    +C_N (1-\lambda_N)^t
    \v^{(N)}\,,
    \label{eq:HeatSol}
\end{align}
where constants~$C_{j}$ depend on the initial condition~$\u(0)$.
It is interesting to note that in Eqn.~\ref{eq:HeatSol}, the
dependence of the solution on higher eigenvectors and
eigenvalues of the Laplacian decays with increasing iteration
count. Thus, it is difficult to devise a fast and distributed
method for clustering graphs based on the heat equation. 

In~\cite{Cit:cluster1,Cit:cluster2},
a novel algorithm based on the idea of permanent
excitation of the eigenvectors of~$\I-\L$ using dynamical systems theory is constructed. In a theme similar to Mark Kac's question ``Can one hear the
shape of a drum?''~\cite{DrumShape}, it was demonstrated that by
evolving the wave equation in the graph, nodes can ``hear'' the
eigenvectors of the graph Laplacian using only local
information. Moreover, it was shown, both theoretically and on
examples, that the wave equation based algorithm is orders of
magnitude faster than random walk based approaches for graphs
with large mixing times. The overall idea of the wave equation
based approach is to simulate, in a distributed fashion, the
propagation of a wave through the graph and capture the
frequencies at which the graph ``resonates''. In other words, it was
shown that by using these frequencies one can compute the
eigenvectors of the Laplacian, thus clustering the graph.

The wave equation based clustering approach can be described as follows. 
Analogous to the heat equation case (Eq.~\ref{eq:HeatSol}), the
solution of the wave equation can be expanded in terms of the
eigenvectors of the Laplacian. However, unlike the heat equation
where the solution eventually converges to the first eigenvector
of the Laplacian, in the wave equation all the eigenvectors
remain eternally excited (a consequence of the
second derivative of $u$ with respect to time). This
observation is used to develop a simple, yet powerful, distributed
eigenvector computation algorithm. The algorithm involves
evolving the wave equation on the graph and then computing the
eigenvectors using local FFTs. The graph
decomposition/partitioning algorithm based on the discretized
wave equation on the graph, given by
\begin{align}
    \u_{i}(t) = 2\u_{i}(t-1) - \u_{i}(t-2) - c^2\displaystyle\sum_{j\in\mathcal{N}(i)}\L_{ij}
    \u_{j}(t-1)\,,
    \label{onenodewave}
\end{align}
where $\sum_{j\in\mathcal{N}(i)}\L_{ij}\u_{j}(t-1)$ originates
from the discretization of the spatial derivatives in the wave equation. The rest of the terms originate
from discretization of the $\partial^{2} u/\partial t^{2}$ term in the wave equation. To
update~$\u_{i}$ using Eq.~\ref{onenodewave}, one needs only the
value of~$\u_{j}$ at neighboring nodes and the connecting edge
weights (along with previous values of~$\u_{i}$).

The main steps of the algorithm are shown as
Algorithm~\ref{alg:WaveAlg}. Note that at each node (node~$i$ in
the algorithm) one only needs nearest neighbor weights~$\L_{ij}$
and the scalar quantities $\u_{j}(t-1)$ also at nearest
neighbors. We emphasize, again, that $\u_{i}(t)$ is a scalar
quantity and \texttt{Random}($[0,1]$) is a random initial
condition on the interval $[0,1]$. The vector~$\v^{(j)}_{i}$ is
the $i$-th component of the $j$-th eigenvector, $T_{max}$ is a
positive integer derived in~\cite{Cit:cluster1,Cit:cluster2},
$\texttt{FrequencyPeak(Y,j)}$ returns the frequency at which the
$j$-th peak occurs and $\texttt{Coefficient}(\omega_{j})$ return the
corresponding Fourier coefficient.

\begin{algorithm}[th!]
    \caption{Wave equation based eigenvector computation algorithm for node~$i$. At node~$i$ one computes the sign of the~$i$-th component of the first~$k$ eigenvectors. The cluster assignment is obtained by interpreting the vector of $k$ signs as a binary number.}\label{alg:WaveAlg}
    \begin{algorithmic}[1]
    \newcounter{ALC@line}
    \setcounter{ALC@line}{1}
    \State $\u_{i}(0) \leftarrow \texttt{Random}\:([0,1])$
    \setcounter{ALC@line}{2}
    \State $\u_{i}(-1) \leftarrow \u_{i}(0)$
    \setcounter{ALC@line}{3}
    \State $t\leftarrow 1$
    \setcounter{ALC@line}{4}
    \While{$t < T_{max}$}
    \setcounter{ALC@line}{5}
         \State \begin{tabular}{l}
         $\u_{i}(t) \leftarrow 2\u_{i}(t-1)-\u_{i}(t-2) -$\\
         $\qquad \qquad c^2 \sum_{j\in\mathcal{N}(i)}\L_{ij}\u_{j}(t-1)$
         \end{tabular}
         \setcounter{ALC@line}{6}
         \State $t\leftarrow t+1$
         \setcounter{ALC@line}{7}
    \EndWhile
    \setcounter{ALC@line}{8}
    \State $Y\leftarrow \texttt{FFT}\:(\left[\u_{i}(1),\dots\dots,\u_{i}(T_{max})\right])$
    \setcounter{ALC@line}{9}
    \For{$j \in \{1,\dots,k\}$}
    \setcounter{ALC@line}{10}
        \State $\omega_{j} \leftarrow
        \texttt{FrequencyPeak}\:(Y,j)$
         \setcounter{ALC@line}{11}
        \State $\v^{(j)}_{i} \leftarrow \texttt{Coefficient}(\omega_{j})$ 
         \setcounter{ALC@line}{12}
        \If {$\v^{(j)}_{i} > 0$}
        \setcounter{ALC@line}{13}
            \State $A_j \leftarrow 1$
            \setcounter{ALC@line}{14}
        \Else
        \setcounter{ALC@line}{15}
            \State $A_j \leftarrow 0$
            \setcounter{ALC@line}{16}
        \EndIf
        \setcounter{ALC@line}{17}
        \EndFor
        \setcounter{ALC@line}{18}
    \State ClusterNumber $\leftarrow \sum_{j=1}^k A_j 2^{j-1}$
\end{algorithmic}
\end{algorithm}

\begin{proposition}
The clusters of graph~$\mathcal{G}$, determined by the signs of
the elements of the eigenvectors of~$\L$, can be computed using
the frequencies and coefficients obtained from the Fast Fourier
Transform of $(\u_{i}(1),\dots,\u_{i}(T_{max}))$, for all~$i$
and some~$T_{max}>0$. Here~$\u_i$ is governed by the wave
equation on the graph (shown in Eqn.~\ref{onenodewave}) with the initial condition~$\u(-1)=\u(0)$
and $0<c < \sqrt{2}$.
\end{proposition}
\begin{proof}
For the proofs see~\cite{Cit:cluster1,Cit:cluster2}.
\end{proof}

The above proof demonstrates that the approach is fundamentally \emph{decentralized}. Moreover, it is shown in~\cite{Cit:cluster1,Cit:cluster2} that the convergence of the wave equation based eigenvector
computation depends on the mixing time of the underlying Markov
chain on the graph, and is given by,
\begin{equation}
    T_{max} = O\left(\arccos\left(\cfrac{2+c^2 (e^{-1/\tau}-1)}{2}\right)^{-1} \right) +
    O(N)\,, \label{eq:Waveconvfinal}
\end{equation}
where $\tau$ is the mixing time of the Markov chain. Thus, the wave equation based algorithm has better scaling with~$\tau$ for graphs of any size (given by $N$, see Fig.~\ref{ConvComp}).

\begin{figure}[t!]
  \centering
  \includegraphics[width=0.85\hsize]{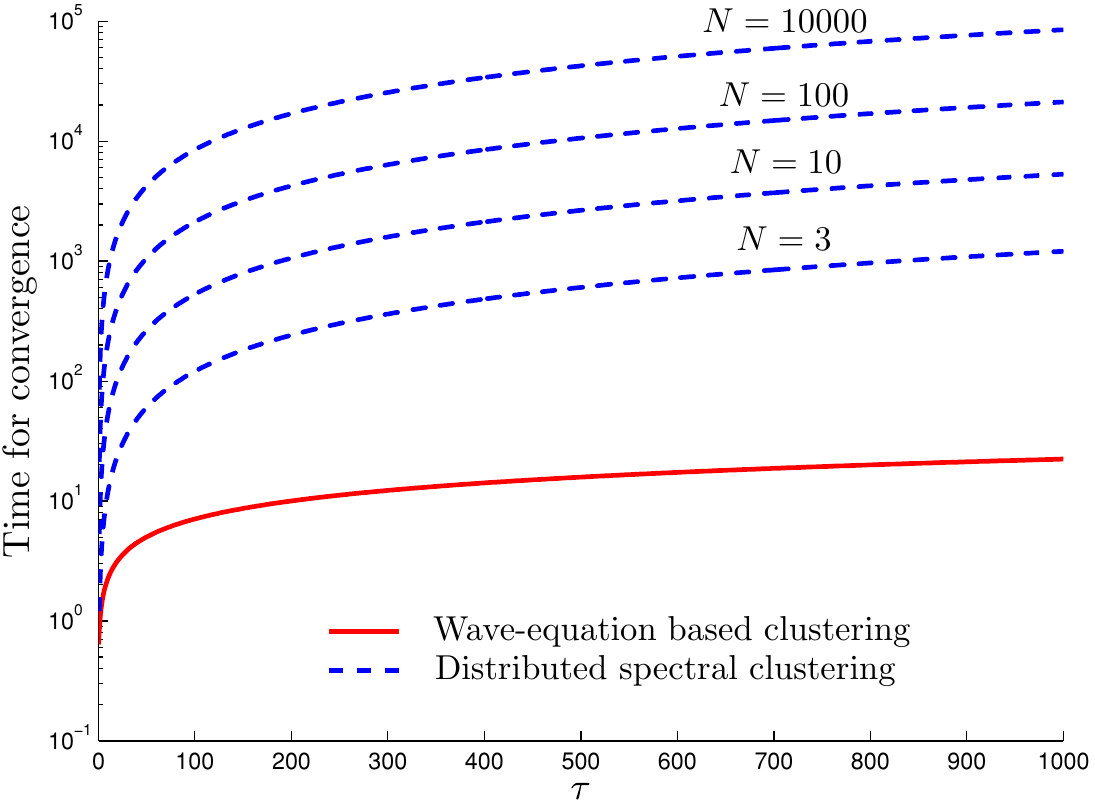}\\
  \caption{Comparison of convergence rates between the distributed algorithm in~\cite{Cit:kempe} and our proposed wave equation
  algorithm for~$c^{2}=1.99$. The wave equation based algorithm has better scaling with~$\tau$ for graphs of any size (given by $N$). The plots are upper bounds on the convergence speed. For more details see~\cite{Cit:cluster2}.\label{ConvComp}}
\end{figure}

The above work is an example of the construction of a state-of-the-art algorithm using dynamical systems theory. This work has also found application in distributed numerical computations~\cite{Cit:Num_klus} and uncertainty quantification~\cite{Cit:UQ_surana}. We now present another example of constructing novel algorithms for  NP-hard problems using the theory of nonlinear dynamics and invariant manifold computations.

\section{Novel algorithm construction: invariant manifolds and the traveling salesman problem}
\label{sec:tsp}
\begin{overview}{Overview}
Recently, dynamical systems theory was used to construct novel algorithms for another iconic NP-hard problem~\cite{Cit:TSP_sahai}. The traveling salesman problem (TSP) has a long and rich history in the areas of computer science, optimization theory, and computational complexity, and has
received decades of interest~\cite{Cit:cook}. This combinatorial optimization problem arises in a wide
variety of applications related to genome map
construction, telescope management, and drilling circuit boards.
The TSP also naturally occurs in applications related to target
tracking~\cite{Cit:target_tracking}, vehicle
routing, and communication networks to name a few. We refer the reader to~\cite{Cit:cook,Cit:TSP_sahai} for further details.

In its basic form, the statement of the TSP is
exceedingly simple. The task is to find the shortest Hamiltonian circuit
through a list of cities, given their pairwise distances. Despite its
simplistic appearance, the underlying problem is NP-hard~\cite{Cit:karp}.
Several heuristics have been developed over the years to solve the problem~\cite{Cit:cook}
including ant colony optimization, cutting plane
methods, Christofides heuristic
algorithm, and the Lin--Kernighan
heuristic. 

In~\cite{Cit:TSP_sahai}, inspired by dynamical systems theory, the authors construct novel orthogonal relaxation based approximations to the TSP. In particular, the constructed dynamical system captures the flow on the manifold of orthogonal matrices and ideally converges to a permutation matrix that minimizes the tour length. However, in general, the flow typically converges to local minima that are not competitive when compared to state-of-the-art heuristics. Inspired by this continuous relaxation, the authors compute the solution to a two-sided orthogonal Procrustes problem~\cite{Cit:Procrustes-book} that relaxes the TSP to the manifold of orthogonal matrices. They then combine the Procrustes approach with the Lin--Kernighan heuristic~\cite{Cit:LKH} for computing solutions of the TSP. Additionally, the authors use set-oriented methods to study the stability of optimal solutions and their stable manifolds, thereby providing insight into the associated basins of attraction and the resulting computational complexity of the problem.
\end{overview}

Given a list of $ n $ cities $ \{ C_{1}, C_{2}, \dots, C_{n} \} $ and the associated distances between cities $ C_i $ and $ C_j $, denoted by $ d_{ij} $, the TSP aims to find an ordering $ \sigma $ of $ \{ 1, 2, \dots, n \} $ such that the tour cost, given by
\begin{equation} \label{eq:basic_cost}
    c = \sum_{i=1}^{n-1} d_{\sigma(i), \sigma(i+1)} + d_{\sigma(n), \sigma(1)},
\end{equation}
is minimized. For the Euclidean TSP, for instance, $ d_{ij} = ||x_i - x_j||_2 $, where $ x_i \in \R^d $ is the position of $ C_i $. In general, however, the distance matrix $ D = (d_{ij}) $ does not have to be symmetric (for example see~\cite{Cit:tsp_asymm}). The ordering $ \sigma $ can be represented as a unique permutation matrix $ P $. Note, however, that due to the underlying cyclic symmetry, multiple orderings -- corresponding to different permutation matrices -- have the same cost.

There are several equivalent ways to define the cost function of the TSP. The authors restrict themselves to the trace\footnote{The trace of a matrix $ A \in \R^{n \times n} $ is defined to be the sum of all diagonal entries, i.e., $ \mbox{tr}(A) = \sum_{i=1}^n a_{ii} $.} formulation. Let $ \mathcal{P}_n $ denote the set of all $ n \times n $ permutation matrices, then the TSP can be written as a combinatorial optimization problem of the form
\begin{equation} \label{eq:TSPCost}
    \min_{P \in \mathcal{P}_n} \mbox{tr} \left( A^T P^T B P \right),
\end{equation}
where $ A = D $ and $ B = T $. Here, $ T $ is defined to be the adjacency matrix of the cycle graph of length $ n $.

One uses the undirected cycle graph adjacency matrix for symmetric TSPs and the one corresponding to the directed cycle graphs for asymmetric TSPs. The matrices are defined as,
\begin{equation*} \label{eq:tmatrix}
    T_\text{dir} = \begin{pmatrix}
        0 & 1 &   &  &   \\
          & 0 & 1 &  &   \\
          &   & \ddots & \ddots &  \\
          &   &   & 0  & 1 \\
        1 &   &   &  & 0
    \end{pmatrix} \; \text{ or } \;
    T_\text{undir} = \begin{pmatrix}
        0 & 1 &   &  & 1 \\
        1 & 0 & 1 &  &   \\
          & \ddots  & \ddots & \ddots &  \\
          &   & 1 & 0  & 1 \\
        1 &   &   & 1 & 0
    \end{pmatrix}.
\end{equation*}
The work in~\cite{Cit:TSP_sahai} focuses on the undirected version of the TSP. By relaxing the TSP problem to the manifold of orthogonal matrices (since permutation matrices are orthogonal matrices restricted to $0$ or $1$ entries), one can use the two sided Procrustes problem to solve the problem exactly, as outlined in the theorem below.

\begin{theorem}
Given two symmetric matrices $ A $ and $ B $, whose eigenvalues are distinct, let $ {A = V_A \Lambda_A V_A^T} $ and $ B = V_B \Lambda_B V_B^T $ be eigendecompositions, with $ \Lambda_A = \diag\left(\lambda_{A}^{(1)}, \dots, \lambda_{A}^{(n)}\right) $, $ \Lambda_B = \diag\left(\lambda_{B}^{(1)}, \dots, \lambda_{B}^{(n)}\right) $, and $ \lambda_{A}^{(1)} \geq \dots \geq \lambda_{A}^{(n)} $ as well as $ \lambda_{B}^{(1)} \geq \dots \geq \lambda_{B}^{(n)} $. Then every orthogonal matrix $ P^* $ which minimizes  
\begin{equation}
\min_{P \in \mathcal{O}_n} || A - P^T B P ||_F
\label{eq:RelTSPCost}
\end{equation} 
has the form
\begin{equation*}
    P^* = V_B S V_A^T,
\end{equation*}
where $ S = \diag(\pm 1, \dots, \pm 1) $.
\end{theorem}

A proof of this theorem can be found in~\cite{Cit:Sch68}. If the eigenvalues of $ A $ and $ B $ are distinct, then there exist $ 2^n $ different solutions with the same cost. If one or both of the matrices possess repeated eigenvalues, then the eigenvectors in the matrices $V_A$ and $V_B$ are determined only up to basis rotations, which further increases the size of the solution space. The Procrustes problem is related to a dynamical system formulation of the TSP as outlined below.

The orthogonal relaxation of the combinatorial optimization problem \eqref{eq:TSPCost}, given by \eqref{eq:RelTSPCost}, can be solved using a steepest descent method on the manifold of orthogonal matrices. For more details about this formulation see~\cite{Cit:TSP_sahai}. One can pose the TSP as a constrained optimization problem of the form,
\begin{align}\label{eq:minPwithEqualityConstraints}
	\min_{P \in \mathcal{O}_n} & \; \mbox{tr} \left( A^T P^T B P \right), \\
	s.t. \; & \; G(P) = 0.
\end{align}
This formulation gives rise to the following set of equations,
\begin{equation}\label{eq:TSP flow}
\begin{aligned}
    \dot{P}       &= -P \left( \left\{ P^T B P, A \right\} + \left \{P^T B^T P, A^T \right\} \right)
- \lambda P \left( (P \circ P)^T P - P^T(P \circ P) \right), \\
\dot{\lambda} &= \frac{1}{3} \mbox{tr} \left( P^T \left( P - (P \circ P) \right) \right).
\end{aligned}
\end{equation}
The above set of equations are obtained by using gradient descent on the Lagrangian cost function.
\begin{example} \label{ex:P flow}
	In order to illustrate the gradient flow approach, let us consider a simple TSP with 10 cities. Using~\eqref{eq:TSP flow}, we obtain the results shown in Figure~\ref{fig:P flow}. In this example, the dynamical system converges to the optimal tour.
	
	\begin{figure}[!htb]
		\centering
		\begin{minipage}[c]{0.32\textwidth}
			\centering
			\includegraphics[width=0.8\textwidth]{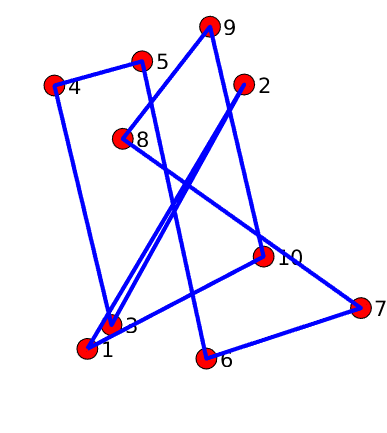}
		\end{minipage}
		\begin{minipage}[c]{0.32\textwidth}
			\centering
			\includegraphics[width=0.8\textwidth]{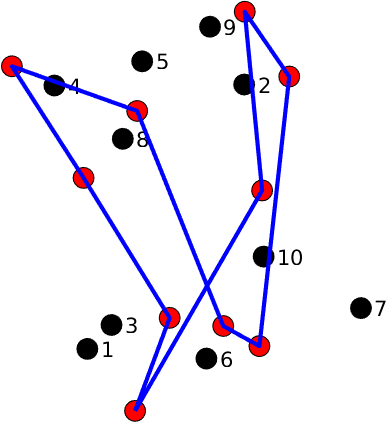}
		\end{minipage}
		\begin{minipage}[c]{0.32\textwidth}
			\centering
			\includegraphics[width=0.8\textwidth]{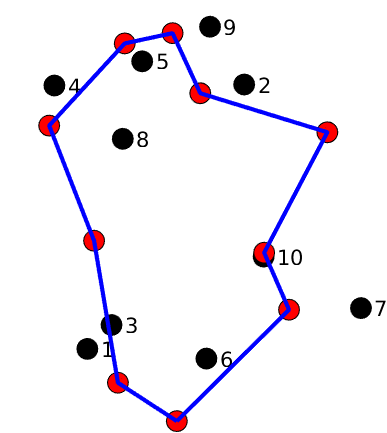}
		\end{minipage} \\[1em]
		\begin{minipage}[c]{0.32\textwidth}
			\centering
			\includegraphics[width=0.8\textwidth]{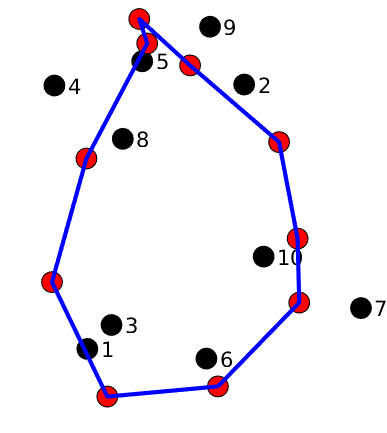}
		\end{minipage}
		\begin{minipage}[c]{0.32\textwidth}
			\centering
			\includegraphics[width=0.8\textwidth]{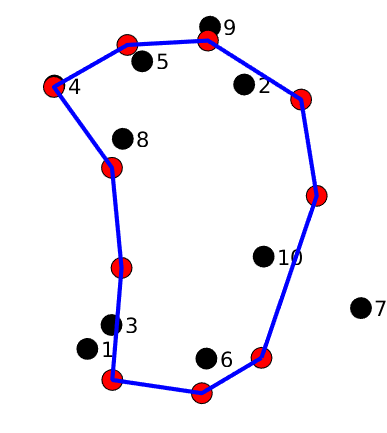}
		\end{minipage}
		\begin{minipage}[c]{0.32\textwidth}
			\centering
			\includegraphics[width=0.8\textwidth]{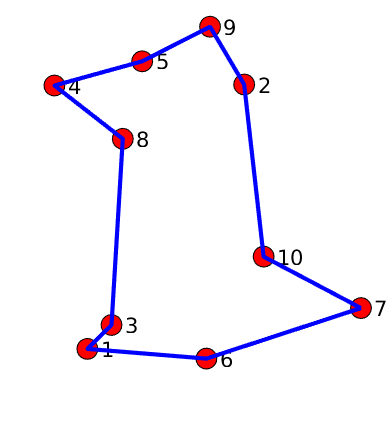}
		\end{minipage}
		\caption{Traveling salesman problem with 10 cities solved using the gradient flow~\eqref{eq:TSP flow}. The original positions of the cities are shown in black, the positions transformed by the orthogonal matrix $ P $ in red. a) Initial trivial tour given by $ \sigma = (1, \dots, 10) $. b--d) Intermediate solutions. e) Convergence to an orthogonal matrix which is ``close'' to a permutation matrix with respect to any matrix norm. f) Extraction of the corresponding permutation matrix. The initial tour was transformed into the optimal tour by the gradient flow.}
		\label{fig:P flow}
	\end{figure}
\end{example}

The dynamical system without constraints converges to equilibria that are given by the Procrustes solutions. To shed light on the stability and local dynamics around the optimal TSP solutions one can approximate \emph{subsets of the stable manifold} of the Procrustes solutions such that two permutation matrices are inside these sets. This numerical study enables the analysis of the robustness of Procrustes solutions under small perturbations of the initial permutation matrix and the assessment of the `closeness' the Procrustes solution is to the optimal permutation matrix. In order to compute the sets of interest, set-oriented continuation techniques developed in \cite{DH96} are used in~\cite{Cit:TSP_sahai}. An example computation is depicted in Fig.~\ref{fig:mostAttractingSets}.  

\begin{figure}[!htb]
	\centering
	\includegraphics[width=.8\textwidth]{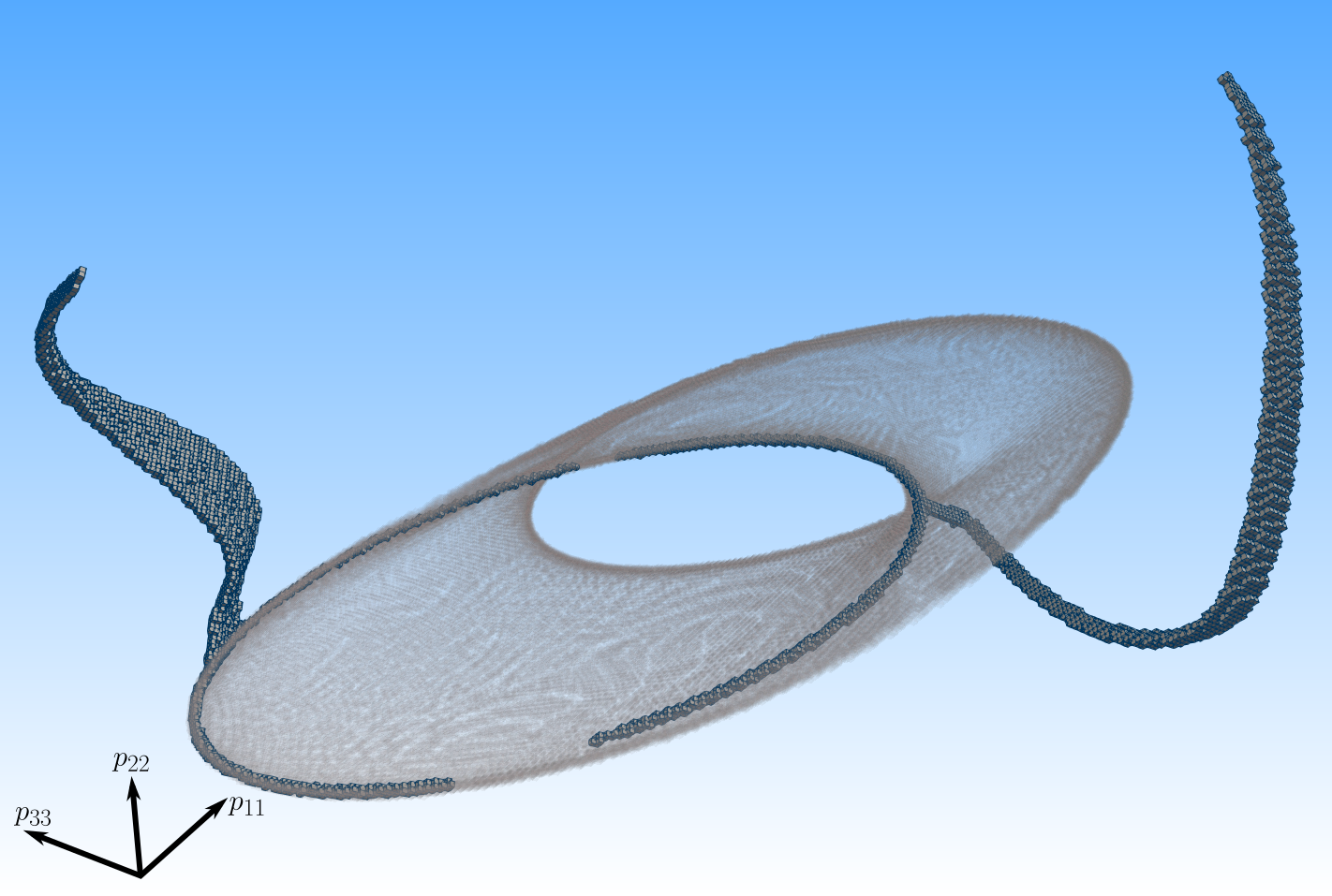}
	\caption{Three-dimensional projection of two subsets of the stable manifold. The omega-limit sets of a small neighborhood of the permutation matrices $P_1$ and $P_2$ form a half circle on their corresponding Procrustes set.}
	\label{fig:mostAttractingSets}
\end{figure}

Moreover, one can also use set oriented methods to compute basins of attraction of optimal permutation matrices for small instances of the TSP. These basins (subsets of the stable manifold) are computed by perturbing the optimal solutions and integrating the flow backward in time~\cite{Cit:TSP_sahai}. The solutions are shown in Fig.~\ref{fig:polar}. These computations are interesting and capture the ``hardness'' of the problem. In particular, one can see that the solutions of relaxed versions of the problem (such  as the relaxations to the manifolds of orthogonal matrices) do not, in general, lie in the basin of attraction of the optimal solutions of the original problem. Other such instances of analysis of relaxed solutions of the TSP using dynamical systems theory are outlined in~\cite{Cit:TSP_sahai}.

\begin{figure}[!htb]
	\begin{minipage}{0.49\textwidth}
		\includegraphics[width = \textwidth]{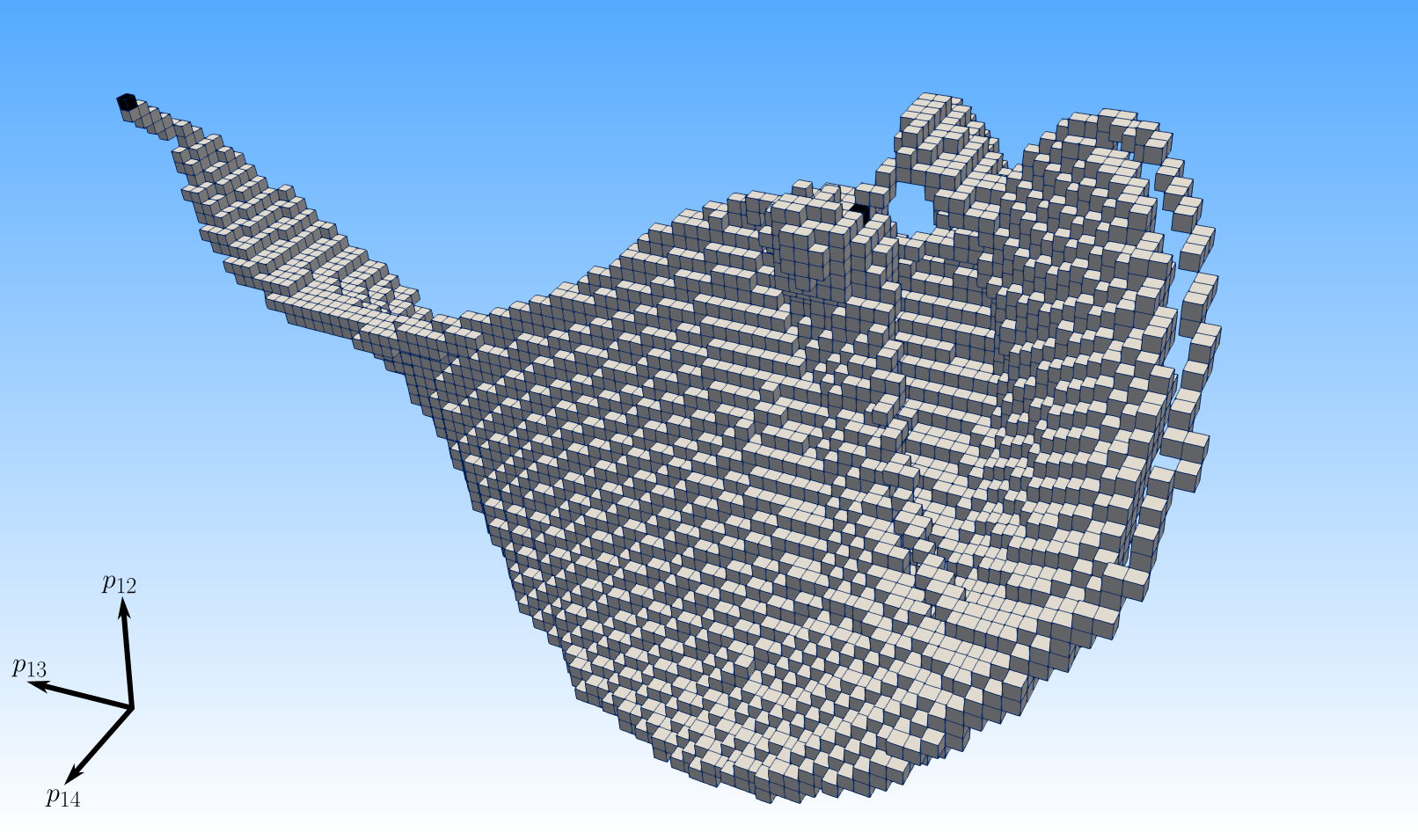}
		\centering \scriptsize{(a)}
	\end{minipage}
	\begin{minipage}{0.49\textwidth}
		\includegraphics[width = \textwidth]{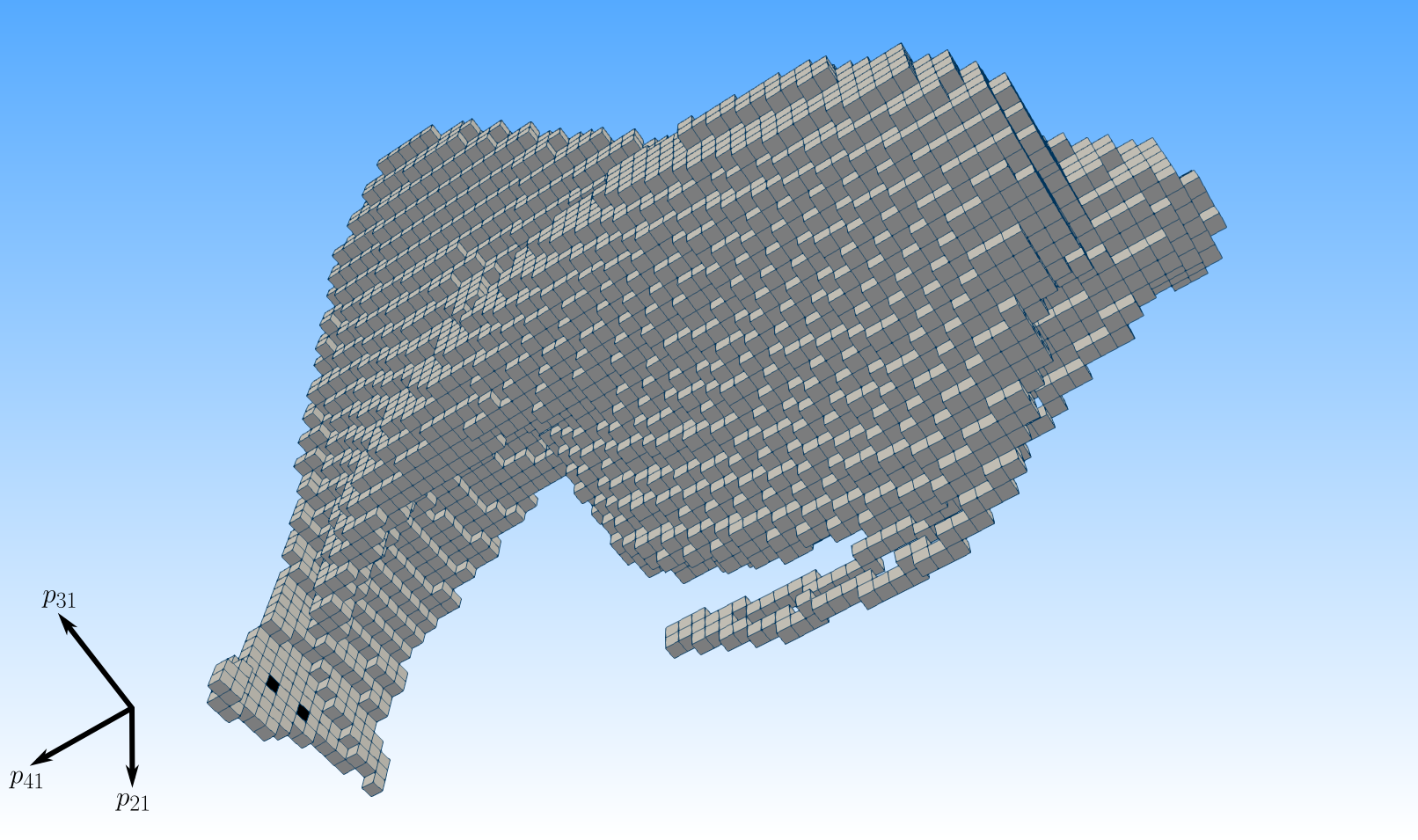}
		\centering \scriptsize{(b)}
	\end{minipage}

	\caption{(a)--(b) Three-dimensional projections of the basin of attraction of $(\widetilde P,\widetilde{\lambda})$. The dark cells depict the stationary solutions of the gradient flow \eqref{eq:TSP flow} backward in time.}
	\label{fig:polar}
\end{figure}

Although, dynamical systems theory demonstrates that the Procrustes solutions do not typically lie in the basin of attraction of the optimal solutions of the TSP, a new biasing scheme for the Lin--Kernighan heuristic is constructed using the aforementioned relaxation~\cite{Cit:TSP_sahai}.

The Lin--Kernighan heuristic is a popular heuristic for the TSP~\cite{Cit:LKH}. Starting from an initial tour, the approach progresses by extracting edges from the tour and replacing them with new edges, while maintaining the Hamiltonian cycle constraint. If $k$ edges in the tour are simultaneously replaced, this is known as the $k$-opt move~\cite{Cit:LKH}. To prune the search space, the algorithm relies on minimum spanning trees to identify edges that are more likely to be in the tour. This ``importance'' metric for edges is called $\alpha$-nearness and described in~\cite{Cit:LKH,Cit:TSP_sahai}. The algorithm has found great success on large instances of the TSP, see~\cite{Cit:cook} for more details. 

In~\cite{Cit:TSP_sahai}, the $\alpha$-nearness metric is replaced with a new Procrustes solution--based metric that prunes/identifies important potential edges to include in the ``candidate set list''. This list is then used to generate the $k$-opt moves. The metric is captured in Fig.~\ref{fig:distance_vs_procrustes}. The Procrustes solution tends to capture the longer edges that are important. To increase the inclusion of the short edges, the approach in~\cite{Cit:TSP_sahai} constructs a homotopy between the Procrustes ($P$-nearness) solution and the distance matrix. Using a graph Laplacian approach, the mixture of the two matrices is compared to the $\alpha$-nearness approach on 22 well-known instances of the TSP.  $ P $-nearness based LKH converges to lower cost values in $18$ of the instances when compared to $\alpha$-nearness based LKH. Moreover, for $50$ random TSP instances of size $1000$  (cities) it is found that $P$-nearness has lower tour costs after a fixed number of $k$-opt moves in $31$ of the instances, translating into an improvement for $62\%$ of the instances.

\begin{figure}[htb]
    \centering
    \includegraphics[width=0.95\textwidth]{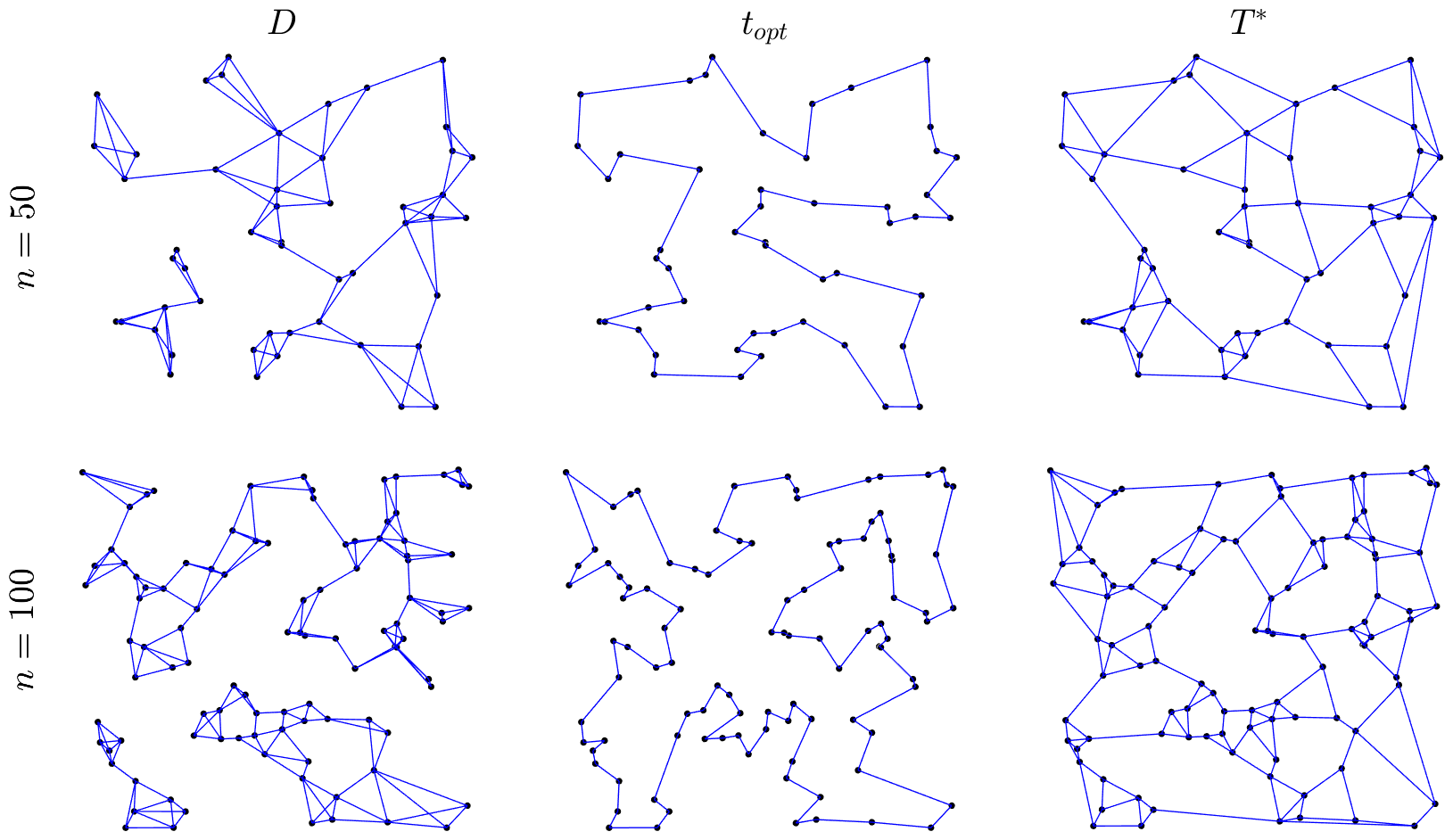}
    \caption{Illustration of $ P $-nearness for random TSP instances of size $50$ and $100$. The left column contains the edges with shortest distance, the center column has the optimal tour for the instances, and the right column contains the edges with the highest $ P $-nearness values for each city. For each city, we plotted the three edges with the highest nearness values.}
    \label{fig:distance_vs_procrustes}
\end{figure}

Thus, this is another example that demonstrates that dynamical systems theory can be used to analyze NP-hard problems and construct improved heuristics.

We now show how networks of Duffing oscillators can be used to construct a new algorithm for the iconic MAX-CUT problem.
\section{Novel algorithm construction: network of Duffing oscillators for the MAX-CUT problem}
\label{sec:maxsat_quantum}
\begin{overview}{Overview}
MAX-CUT~\cite{Cit:karp} is a well-known NP-hard problem that arises in graph theory. Simply stated, the goal is to compute a subset $S$ of the vertex set in a graph $\mathcal{G}$, such that the number of edges between $S$ and the rest of the graph are maximized. The best known approximation ratio of $0.878$ can be achieved in polynomial time using semi-definite programming~\cite{Cit:Goemans}. The problem naturally arises in VLSI design and statistical physics, and has been extensively studied.

In~\cite{Cit:quantum_net}, the authors construct an optimization algorithm by simulating the adiabatic evolution of Hamiltonian systems which can be used to approximate the MAX-CUT solution of an all-to-all connected graph. The approach is inspired by quantum adiabatic optimization for Ising systems~\cite{Cit:Ising} with the following energy,
\begin{align}
    E_{Ising}(s) = -\frac{1}{2}\displaystyle\sum_{i=1}^{N}\sum_{j=1}^{N}J_{ij}s_{i}s_{j},
    \label{eq:ising}
\end{align}
where $s_i$ are the spins which can take values \{-1,1\} and $J_{ij}$ is the coupling coefficient. Finding the lowest energy state of the Ising system is computationally challenging (for a system with N spins, the potential number of states is $2^N$). Note that one can map the Ising problem to the MAX-CUT problem by setting $J_{ij}=-w_{ij}$, where $w_{ij}$ is the weight of the edge that connects nodes $i$ and $j$. It is easy to show that minimizing the energy in Eqn.~\ref{eq:ising} corresponds computing the solution of the MAX-CUT problem.

The approach outlined in~\cite{Cit:quantum_net} relies on the adiabatic evolution of a network of nonlinear oscillators. This system exhibits a bifurcation (called ``simulated bifurcation'') for each nonlinear oscillator. The two branches correspond to the $-1$ and $+1$ values for each spin. The authors exploit Graphical Processing Units (GPU) and Field Programmable Gate Array (FPGA) platforms to compute the solution of these Hamiltonian systems. This method is compared to existing methods and displays orders-of-magnitude improvement. The approach is demonstrated on an Ising system with 100,000 spins. 
\end{overview}
Consider the Hamiltonian that arises in Kerr-nonlinear parametric oscillators,
\begin{align}
    H(x,y,t) &= \displaystyle\sum_{i=1}^{N}\left[\frac{K}{4}(x_{i}^2 + y_{i}^2)^2  - \frac{p(t)}{2}(x_{i}^2 - y_{i}^2) + \frac{\Delta_{i}}{2}(x_{i}^2 + y_{i}^2)\right] \nonumber\\
     &-\frac{\xi_{0}}{2}\displaystyle\sum_{i=1}^N\sum_{j=1}^N J_{ij}(x_i x_j + y_i y_j).
\label{eq:ham}
\end{align}
Here $x_i$ and $y_i$ are position and momentum of the $i$-th oscillator respectively, $K$ is the Kerr coefficient, $p(t)$ is the parametric pumping amplitude, and $\Delta_i$ is the detuning frequency between the natural frequency of the $i$-th oscillator and half the pumping frequency. Using the standard Hamiltonian formulation, one can derive equations of motion for each oscillator. Evolving these system of equations for $x_i$ and $y_i$ converges to low energy solutions of an Ising system (Eqn.~\ref{eq:ising}) with high probability. Thus, the sign of $x_i$ at the end of the simulation determines the $i$-th spin. However, the above equations are computationally challenging to simulate from a numerical perspective. 

Instead of using the equations that arise from the ``full'' Hamiltonian in Eqn.~\ref{eq:ham}, the authors (in~\cite{Cit:quantum_net}) construct a simplified Hamiltonian of the form,
\begin{align}
    H(x,y,t) &= \displaystyle\sum_{i=1}^{N} \frac{\Delta}{2}y_i^2 + V(x,t) \nonumber\\
    &= \displaystyle\sum_{i=1}^{N} \frac{\Delta}{2}y_i^2 + \left[\frac{K}{4}x_{i}^4 + \frac{\Delta - p(t)}{2}x_{i}^2\right] - \frac{\xi_{0}}{2}\displaystyle\sum_{i=1}^N\sum_{j=1}^N J_{ij}x_i x_j.
    \label{eq:ham2}
\end{align}
The above Hamiltonian corresponds to the following system of equations,
\begin{align}
    \dot x_{i} &= \Delta y_i \nonumber\\
    \dot y_{i} &= -\left[Kx_{i}^2 -p(t) + \Delta\right]x_{i} + \xi_{0} \displaystyle\sum_{i=1}^N\sum_{j=1}^N J_{ij}x_j.
    \label{eq:duffing_sys}
\end{align}
It is easy to see that the above system of equations are a network of Duffing oscillators~\cite{Cit:Gucken}. The separability of the Hamiltonian (Eqn.~\ref{eq:ham2}) makes the numerical integration of the system of equations easier. In particular, the authors use an explicit symplectic Euler scheme which makes it amenable for one to hard wire the resulting computational circuits on an FPGA platform. The computation proceeds as follows: all $x$ and $y$ variables are initially set to zero, $p(t)$ is then increased from $0$ and the system in Eqn.~\ref{eq:duffing_sys} is evolved. The sign of the final value of $x_i$ serves as an approximation of the $i$-th spin of the associated Ising system. 

The system in Eqn.~\ref{eq:duffing_sys} has two branches of solutions as $p(t)$ is increased from zero. It is easy to see that these branches correspond to $\pm\sqrt{p-\Delta/K}$ for each oscillator and, consequently, leads to a $2^N$ solution space. If one varies $p(t)$ slowly, the adiabatic theorem ensures that if one converges to a low energy solution for $p(t)$ close to $0$, the final solutions (for large $p(t)$) will also correspond to low energy.

This method is compared to state-of-the-art approaches for two instances of the MAX-CUT problem. In the first instance, an all-to-all 2000 node MAX-CUT problem is solved on an FPGA using the above approach. The authors demonstrate that the above framework successfully converges to the best known solutions~\cite{Cit:Goemans} very quickly. Moreover, they test the approach on a 100,000 size problem (with $5\times10^9$ edges) and find that their approach converges to the answer $100-1000$ times faster than existing software on GPU hardware. For more details of the work and associated results we refer the reader to~\cite{Cit:quantum_net}.

Thus far, we have summarized three examples in which dynamical systems theory was used to construct novel algorithms for NP-hard problems. We now discuss approaches that exploit nonlinear dynamics theory to analyze optimization algorithms for NP-hard problems.

\section{Analysis of algorithms: Koopman operators based analysis of algorithms}
\label{sec:koopman}
\begin{overview}{Overview}
Koopman operator theory is one of the most active and exciting sub-areas within dynamical systems theory~\cite{Cit:koop_mez,Cit:koop_klus, Cit:koop_kev,Cit:koop_kutz}. The approach is based on the construction of an infinite dimensional linear operator that captures the evolution of the observables of the underlying nonlinear system. Consequently, the spectra and eigenfunctions of the operator, capture system dynamics. This methodology has been used used in a wide variety of settings, including system control and identification. An advantage of Koopman operator based methodologies is that the computations are typically based on time trace data of system evolution~\cite{Cit:DMD}.  In recent work~\cite{Cit:koopman}, Koopman operator theory was used to analyze algorithms. In particular, the authors consider optimization algorithms that evolve their state in the form of iterations. An assumption is made that the algorithm state spaces $X\subseteq \mathbb{R}^d$ are smooth $k$-dimensional Riemannian submanifolds in $d$-dimensional Euclidean spaces. A single iteration of the state $x_n$ is represented as,
\begin{equation}
    x_{n+1} = a(x_n), n\in\mathbb{N},
\end{equation}
where $n$ is the iteration count. These iterative algorithms can sometimes be represented in continuous form (akin to the process used in~\cite{Cit:Nesterov}). In other words, one can (in the limit) represent the algorithm as a continuous vector field $v:X\rightarrow\mathbb{R}^k$. If one starts at an initial condition $x_0$, the continuous time representation is of the form~\cite{Cit:koopman},
\begin{equation}
    \frac{dc(s)}{ds} = v(c(s)),\quad c(0) = x_{0}.
\end{equation}
As shown in~\cite{Cit:koopman}, this continuous form can be approximated using a Koopman operator framework. The authors then use this approach to study gradient descent and Newton-Raphson from a global dynamics perspective. Although the examples fall under the category of continuous optimization, the approach can certainly be used to study combinatorial optimization algorithms in the future.
\end{overview}
Using the same nomenclature as in~\cite{Cit:koopman}, consider a dynamical system of the form, 
\begin{equation}
    \frac{dS_{t}(x_0)}{dt} = v(S_{t}(x_0)),
    \label{eq:dyn_algo}
\end{equation}
then the family of Koopman operators $\mathscr{K}^t$ acts on the function space of observables $g:X\rightarrow\mathbb{C}$ as follows,
\begin{equation}
    [\mathscr{K}^t g](x) = (g\circ S_{t})(x).
    \label{eq:koopman}
\end{equation}
An $L^2$ function space with an inner product is typically chosen for the space of observables. For more details, on the approach and choice of function space see~\cite{Cit:koopman}. Note that the Koopman operator is the adjoint of the Perron-Frobenius operator~\cite{Cit:Perron_Frob}. 
Letting $t=1$, without loss of generality, The Koopman operator can be expanded in terms of its spectrum,
\begin{equation}
    \mathscr{K} = \displaystyle\sum_k\lambda_k P_{\lambda_k} + \int_{\sigma_{ac}}\lambda dE(\lambda)
    \label{eq:spectrum}
\end{equation}
where $\lambda_k$ lie in the discrete part and $\sigma_{ac}$ is the continuous spectrum of the operator. $P_{\lambda}$ and $dE(\lambda)$ are projection operators for their corresponding eigenspaces. The eigenfunctions of the Koopman operator are,
\begin{equation}
    [\mathscr{K}^t \phi_{\lambda}](x) = (\phi_{\lambda}\circ S_{t})(x) = \lambda^t\phi_{\lambda}(x).
\end{equation}
Thus, one can predict the evolution of observables,
\begin{equation}
    \mathscr{K}^t g = \displaystyle\sum_{k}c_{k}\lambda^t\phi_{\lambda,k}.
    \label{eq:observerable}
\end{equation}

The operator can be numerically approximated using a data driven approach as outlined in~\cite{Cit:koop_kev,Cit:koop_kutz}. The popular extended dynamic mode decomposition (EDMD) methodology introduced in~\cite{Cit:koop_kev} is used to analyze algorithms using the Koopman operator lens~\cite{Cit:koopman}. 

The EDMD approach approximates the action of the infinite dimensional operator using a finite set of real-valued functions (also called ``dictionary''). In particular, given a smooth manifold $M \subset \mathbb{R}^d$ sampled by a finite set of points $X=\{x_{i}\in M\}$, the EDMD approach computes the action of the Koopman operator on the dictionary of points in $X$. The operator itself is approximated using a least squares approach~\cite{Cit:koop_kev} as outlined below. Given a dictionary of $N_D$ observables $D=\{d_i:M\rightarrow\mathbb{R}|\,\,i=1,\hdots,N_{D}\}$ one can define a matrix of the form $G=[d_{1}(X),d_{2}(X),\hdots,d_{N_D}(X)]$. Then the Koopman operator $\mathscr{K}^t$ can be approximated as,
\begin{equation}
    K = \frac{1}{N_X^2}(G^{T}G)^{\dagger} (A^{T}A),
\label{eq:koop_num}
\end{equation}
where $N_X$ is the size of the dataset and $A=\left[\mathscr{K}^t d_{1}(X), \mathscr{K}^t d_{2}(X),\hdots,\mathscr{K}^t d_{N_D}(X)\right]$. For more details see~\cite{Cit:koop_kev}.

The operator gives a local approximation of the underlying algorithm applied to a specific instance of a problem. In particular, one can use the data of a short burst of computation to compute a local approximation of the dynamics of the algorithm to \emph{accelerate} its convergence. The eigenvalues, vectors, and modes are computed using EDMD. This approximation is used as a data-driven surrogate for the system to accelerate optimization. In~\cite{Cit:koopman}, the authors use the following cost function,
\begin{equation}
    f(x_{1},x_{2}) = (x_{1}^2 + x_{2} - 11)^2 + (x_{1} + x_{2}^2 -7)^2,
\end{equation}
to demonstrate utility of the Koopman approach. The function has one local maximum and four local minima~\cite{Cit:koopman}. The authors study the the popular gradient descent algorithm using the Koopman operator framework. In particular, they use radial basis functions to form a dictionary and compute $503$ eigenvalues and eigenfunctions. They show that one can construct an ergodic decomposition of the state space, thereby separating the different basins of attraction~\cite{Cit:koopman}. The approach is able to capture the global dynamics of the algorithm in this setting, providing valuable insight regarding the performance and limitations of the algorithm.

Additionally, the authors demonstrate the use of Koopman operators for studying global dynamics of algorithms in high-dimensional spaces. They take the example of a $100$-dimensional problem and show that the dynamics quickly contracts to a low dimensional subset. They demonstrate that one can accelerate the prediction of trajectories of gradient descent with high accuracy. The work concludes with the illustration of the utility of Koopman operators for analyzing the iconic Newton-Raphson method for root finding~\cite{Cit:koopman}. For a complex polynomial of degree two, they show that the eigenfunction diverges at the roots. They also show that in cases of chaotic behavior of Newton-Raphson, one can use approximations~\cite{Cit:Mezic_cont} of the continuous spectrum of the Koopman operator to study statistical properties of the emergent chaos.

Although the above Koopman methodology was demonstrated on problems of continuous optimization, it provides a new technique with which one can study combinatorial optimization problems. We anticipate that the Koopman operator approach will be a new tool with which to study algorithms for NP-hard problems and improve their performance.

\section{Analysis of algorithms: Chaos and dynamical systems for analyzing the satisfiability (SAT) problem}
\label{sec:dyn_sat}
\begin{overview}{Overview}
The satisfiability problem is another iconic problem that frequently arises in the study of computational complexity theory. The challenge here is to find a satisfying assignment for a logical formula. In particular, a $k$-SAT Boolean formula $\phi(x)$ of $N$ Boolean variables and $m$ clauses, $\phi:\{0,1\}^{N}\rightarrow\{0,1\}$, is written
in the conjunctive normal form (CNF)~\cite{biere2009handbook} as follows,
\begin{equation}
\label{eq:sat}
\phi(x) = \bigwedge_{i=1}^m C_i = \bigwedge_{i=1}^m (x_{i_1} \lor x_{i_2} \lor\hdots\lor x_{i_{k}}),
\end{equation}
where $x_{i_l}$ is the $l^{\rm th}$ literal in clause $C_{i}$. A SAT formula is said to be 
\emph{satisfiable} if there exists an assignment for the binary variables $\x$ such that $\phi(\x)=1\,\,          (\text{true})$. It is well known that the satisfiability problem is NP-complete~\cite{Cit:cook}. A critical parameter associated with the satisfiability problem is the clause density $\alpha = m/N$~\cite{biere2009handbook}. In particular, the probability that a random $k$-SAT instance is satisfiable undergoes a phase transition as a function of $\alpha$ ($N\rightarrow\infty$)~\cite{biere2009handbook}.  The MAX-SAT problem (and the corresponding weighted version)~\cite{krentel1988complexity} requires one to find that assignment (or assignments) that maximize the number (or the cumulative weights) of satisfied clauses. Consider a SAT formula $\phi$, then every assignment $x$ can be mapped to an ``energy'' $\Phi(x)$ such that,
\begin{equation}
  \label{eq:1}
  \Phi(\x) = \sum_{i=1}^{m} C_{i},
\end{equation}
where $C_i=1$, if the $i$-th clause evaluates to $\text{true}$. In other words, the goal under the MAX-SAT problem is to find the assignment for $\x$
such that the number of satisfied clauses (or energy) is maximized. The MAX-SAT problem is harder (from a computational standpoint) than the SAT problem. In particular, it is known to be strongly NP-hard (there are no polynomial time approximation schemes). The problem of computing density of states (DOS) encompasses the the SAT and MAX-SAT problems. Classical and quantum algorithms for estimating the DOS of logical formulae were constructed in~\cite{Cit:SAT_sahai}.

In a series of seminal papers~\cite{Cit:zoltan,Cit:zoltan2,Cit:zoltan3}, the authors construct a dynamical systems approach to study satisfiability problems. They construct a dynamical system that computes the solutions of SAT instances. Here, the equilibria of the dynamical system correspond to literal values for which the SAT formula in Eqn.~\ref{eq:sat} evaluates to $\texttt{true}$. They prove that the dynamical system admits no \emph{false} equilibria or limit cycles. Additionally, they relate the emergence of transient chaos and fractal boundaries with optimization hardness of the problem instance~\cite{Cit:zoltan}, pointing to a deep connection between dynamical systems theory and computational complexity.
\end{overview}
%
As mentioned in the overview, in~\cite{Cit:zoltan}, the authors embed SAT equations into a system of ordinary differential equations using the the following mapping: let $s_{x_{i}} = \left[-1,1\right]$, i.e. $s_{x_{i}}$ can take values between $-1$ and $1$ such that,
\begin{align*}
s_{x_i} =
\begin{cases}
-1, & \text{if}\ x_i= \mbox{FALSE}  \\
1, & \text{if}\ x_i= \mbox{TRUE}.
\end{cases}
\end{align*}
Generalizing the dynamical system for satisfiability problems constructed in~\cite{Cit:zoltan}, one can define $c_{mi}$ and ${K_m}$ as follows,
\begin{align*}
c_{mi} =
\begin{cases}
-1, & \text{if}\ s_{x_i}\,\, \text{appears in negated form in m-th clause}  \\
1, & \text{if}\ s_{x_i}\,\, \text{appears in direct form in m-th clause}  \\
0, & \text{otherwise},
\end{cases}
\end{align*}

\begin{align*}
K_m(s) = 2^{-k} \prod_{j=0}^{k-1}\prod_{i=1}^{N} (1-c_{mi}s_{x_i}) \quad \forall m=1,2,\hdots,M.
\end{align*}
Note that $K_m(s)=0$, if and only if $m$-th clause is satisfied i.e. $c_{mi}s_{x_i}=1$ for at least one variable $x_i$ that appears in clause $m$. In~\cite{Cit:zoltan},  the authors define an energy function of the form $V(s)=\sum_{m=1}^{M}a_{m}K_{m}(s)^2$ such that $V(s^{*})=0$ only at a solution $s^{*}$ of the satisfiability problem. The \emph{auxiliary} variables $a_{m}\in (0,\infty)$ prevent the non-solution attractors from trapping the search dynamics (for more information see~\cite{Cit:zoltan}). 

In~\cite{Cit:zoltan}, the authors find that as the constraint density of the k-SAT problem increases, the trajectories of the dynamical system display intermittent chaos with fractal basin boundaries~\cite{Cit:zoltan}. Note that, in this work, the existence of chaos is associated with the emergence of positive finite size Lyapunov exponents (FSLE)~\cite{Cit:gucken} and the emergence of chaos corresponds to the well known phase transitions in the $k$-SAT problem~\cite{Cit:zoltan}.

In~\cite{Cit:zoltan2}, the authors further exploit the above system of equations to study the $k$-SAT problem with increasing constraint density. They find that hardness appears as a second order phase transition and discover that the resulting transient chaos displays a novel exponential-algebraic scaling. In~\cite{Cit:zoltan3}, the authors exploit the above framework to construct novel solvers for the MAX-SAT problem. This body of work demonstrates that dynamical systems theory can, in fact, be used to simultaneously study computational complexity theory and construct novel algorithms for NP-hard problems.

\section{Conclusion}
Combinatorial optimization is a wide and important area of research with numerous applications. For decades, computer scientists have developed novel algorithms for addressing these problems. Some  problems are amenable to algorithms and theory developed thus far (examples include graph routing and sorting), while others, in general, remain intractable from a computational standpoint (such as the traveling salesman problem and MAX-SAT) despite significant efforts. The classification of problems into different classes (such as NP, NP-hard, and PSPACE) and associated analyses has given rise to the field of computational complexity theory.

Nonlinear dynamics, on the other hand, arises in a multitude of engineering and scientific settings. The theory has been used to explain system behavior in a diverse set of fields such as fluidics, structural mechanics, population dynamics, epidemiology, optics, and aerospace propulsion. However, the application of the theory of dynamical systems to combinatorial optimization and computational complexity remains limited. 

In this survey article, we summarize five recent examples of using dynamical systems theory for constructing and analyzing combinatorial optimization problems. In particular, we cover a) a novel approach for clustering graphs using the wave equation partial differential equation (PDE),   b) invariant manifold computations for the traveling salesman problem, c) novel approaches for building quantum networks of Duffing oscillators to solve the MAX-CUT problem, d) applications of the Koopman operator for analyzing optimization algorithms, and e) the use of dynamical systems theory to analyze computational complexity of the SAT problem.

We note that the above set of examples are not comprehensive and there are several examples that have been omitted in this survey. However, the goal of this article is not to provide a complete list of all such examples but to demonstrate that dynamical systems theory can be exploited to construct algorithms and approaches for optimization problems. Even more importantly, we hope to inspire others to extend existing dynamical systems approaches to construct the next-generation of techniques and insights for combinatorial optimization.
\begin{acknowledgement}
This material is based upon work supported by the Defense Advanced Research Projects Agency (DARPA) and Space and Naval Warfare Systems Center, Pacific (SSC Pacific) under Contract No. N6600118C4031.
\end{acknowledgement}
%
%

%
\biblstarthook{}

\end{document}